\documentclass[11pt,leqno]{amsart}

\usepackage{graphicx}
\usepackage{amsmath,amsfonts,amssymb,amsthm}

\usepackage{url}

\usepackage[margin=1in]{geometry}

\usepackage{subcaption}
\usepackage{float}

\usepackage{xspace}

\usepackage{xcolor}

\newtheorem{Thm}{Theorem}
\newtheorem{Prop}[Thm]{Proposition}
\newtheorem*{Prop*}{Proposition}
\newtheorem{Cor}[Thm]{Corollary}
\newtheorem{Lem}[Thm]{Lemma}

\theoremstyle{remark}
\newtheorem*{Remarks}{Remarks}
\newtheorem*{Remark}{Remark}

\theoremstyle{definition}
\newtheorem{Def}[Thm]{Definition}

\numberwithin{equation}{section}
\numberwithin{Thm}{section}

\newcommand{\R}{\mathbb R}
\newcommand{\N}{\mathbb N}
\newcommand{\C}{\mathbb C}

\newcommand{\ft}{\tilde{f}}
\newcommand{\St}{\tilde{S}}
\newcommand{\Ft}{\tilde{F}}
\newcommand{\alphat}{\tilde{\alpha}}
\newcommand{\Ct}{\tilde{C}}
\newcommand{\Conet}{\tilde{C_1}}

\newcommand{\Fkp}{F_{k,p}}

\renewcommand{\Re}{\operatorname{Re}}
\renewcommand{\Im}{\operatorname{Im}}

\begin{document}


\title{H\"older continuity and dimensions of fractal Fourier series}

\author{Efstathios-K. Chrontsios-Garitsis}
\author{A.J. Hildebrand}  
\address{Department of Mathematics\\
University of Illinois\\
Urbana, IL 61801\\
USA}
\email[E.-K. Chrontsios-Garitis]{ekc3@illinois.edu,echronts@gmail.com}
\email[A.J. Hildebrand]{ajh@illinois.edu}

\date{\today}



\maketitle


\begin{abstract}
Motivated by applications in number theory, analysis, and fractal geometry, 
we consider regularity properties and dimensions of graphs associated 
with Fourier series of the form $F(t)=\sum_{n=1}^\infty f(n)e^{2\pi i
nt}/n$, for general coefficient functions $f$.
Our main result states that if, for some
constants $C$ and $\alpha$ with $0<\alpha<1$, we have  
$|\sum_{1\le n\le x}f(n)e^{2\pi i nt}|\le C x^{\alpha}$ uniformly
in $x\ge 1$ and $t\in\R$, then the series $F(t)$ 
is H\"older continuous with exponent $1-\alpha$, and the graph of
$|F(t)|$ on the interval $[0,1]$ has box-counting dimension $\le
1+\alpha$. 
As applications we recover the best-possible uniform H\"older exponents 
for the Weierstrass functions $\sum_{k=1}^\infty a^k\cos(2\pi b^k t)$
and the Riemann function $\sum_{n=1}^\infty \sin(\pi n^2 t)/n^2$.
Moreoever, under the assumption of the Generalized Riemann Hypothesis, 
we obtain nontrivial bounds for H\"older exponents 
and dimensions associated with series of the form
$\sum_{n=1}^\infty \mu(n)e^{2\pi i n^kt}/n^k$, where $\mu$ is the
M\"obius function.
\end{abstract}

\section{Introduction}
\label{sec:introduction}

Given an arithmetic function $f:\N\to\C$, we consider the Fourier series 
\begin{equation}
\label{eq:F-def}
F(t)=F(f;t)=\sum_{n=1}^\infty \frac{f(n)e^{2\pi i nt}}{n}.
\end{equation}
We are interested in regularity properties of the function $F(t)$ and
in the dimensions of the following 
natural geometric objects associated with this
function: 
\begin{itemize}
\item 
The image $F([0,1])$ of the interval $[0,1]$ under $F$, i.e., the path 
$F(t)$, $0\le t\le 1$, in the complex plane (identified with $\R^2$):
\begin{equation}
\label{eq:F01-def}
F([0,1])=F(f;[0,1])=\{(\Re F(t), \Im F(t))\in\R^2: t\in[0,1]\}.
\end{equation}
\item The graphs of the real-valued functions 
$|F(t)|$, $\Re F(t)$, and $\Im F(t)$, 
$0\le t\le 1$, which we denote by $G(F)$, $G_R(F)$, and $G_I(F)$, respectively;
i.e., the sets
\begin{align}
	\label{eq:GF-def}
	G(F)&=G(f;F)=\{ (t,|F(t)|)\in \R^2: t\in[0,1] \} ,
	\\
	\label{eq:GRF-def}
	G_R(F)&=G_R(f;F)=\{ (t,\Re F(t))\in \R^2: t\in[0,1] \} ,
	\\
	\label{eq:GIF-def}
	G_I(F)&=G_I(f;F)=\{ (t,\Im F(t))\in \R^2: t\in[0,1] \} .
\end{align}
\end{itemize}

Motivated by applications in both number theory and analysis,
we are interested in relating uniform bounds for the exponential sums
\begin{align}
	\label{eq:S-def}
	S(x,t)&=S(f;x,t)=\sum_{1 \leq n\leq x}f(n)e^{2\pi i nt}
\end{align}
to regularity conditions on the function $F(t)$ and 
dimensions of the associated geometric objects  defined above.
Our key result is the following theorem, which 
provides such a relation between bounds for $S(x,t)$ 
and the H\"older exponent of the associated Fourier series $F(t)$. 
(See Definition \ref{def:holder} below for the definitions of H\"older
continuity and H\"older exponent.)

\begin{Thm}\label{thm:main}
Assume that $S(x,t)$ satisfies, 
for some constants $C>0$ and $0<\alpha<1$,  
\begin{equation}
		\label{eq:S-bound}
		|S(x,t)|\le Cx^{\alpha}
		\quad (x\ge 1,\ t\in\R).
	\end{equation} 
Then we have
	\begin{equation}
		\label{eq:F-bound}
		|F(t+h)-F(t)|\le C_1h^{1-\alpha}
		\quad (t\in\R,\ h>0),
	\end{equation}
where 
\begin{equation}
\label{eq:C1-formula}
C_1=C_1(C,\alpha)=\frac{6\pi C}{\alpha(1-\alpha)}.
\end{equation}
Moreover, the same bound holds for the functions
$|F(t)|$, $\Re F(t)$, and $\Im F(t)$.

In particular, the functions $F(t)$, $|F(t)|$, $\Re F(t)$, and $\Im
F(t)$ are H\"older continuous with exponent $1-\alpha$.
\end{Thm}

Using known results relating the H\"older exponent of a function to the
dimensions of its graph and image set  (see Proposition
\ref{prop:dimdist} below), this yields the following
corollary.  (See Definition \ref{def:box-dim} below for the definition
of box-counting dimension.)


\begin{Cor}\label{cor:main}
Under the assumptions of Theorem \ref{thm:main} we have:
\begin{itemize}
\item[(i)]
The upper
box-counting dimension of the image $F([0,1])$ of the interval
$[0,1]$ under $F$ satisfies
\begin{align}
		\label{eq:dimGF}
		\overline{\dim_B}(F([0,1]))&\leq\frac{1}{1-\alpha}.
\end{align}
\item[(ii)]
The upper box-counting dimension
of the graph $G(F)$ defined in \eqref{eq:GF-def}
satisfies
	\begin{align}
		\label{eq:dimF}
		\overline{\dim_B}(G(F))&\leq 1+\alpha.
\end{align}
	Moreover, the same bound holds
	for the upper box-counting dimensions of the graphs
	$G_R(F)$ and $G_I(F)$, defined in 
	\eqref{eq:GRF-def} and \eqref{eq:GIF-def}, respectively. 
\end{itemize}
\end{Cor}


A key feature of Theorem \ref{thm:main} is its generality: aside from
the exponential sum estimate \eqref{eq:S-bound}, there are no
restrictions on the  coefficients $f(n)$.  In particular, the function
$f$ need not satisfy any regularity or smoothness properties, nor does
it have to be bounded. This opens up the result to a wide range of
potential applications.   

Another feature of Theorem \ref{thm:main} is the explicit nature of the
constant $C_1(C,\alpha)=6\pi C/(\alpha(1-\alpha))$. We hope that
this may be useful in future applications.


The intuition behind the exponent $1-\alpha$ in \eqref{eq:F-bound} is as
follows: Assume, for simplicity, that the function $f$ is bounded by $1$. Then
we have the trivial bound 
$|S(x, t)| \le  \sum_{1\le n\le x} |f (n)|\le x$, 
so \eqref{eq:S-bound} represents a saving of a factor
$x^{1-\alpha}$ over this trivial bound. On the other hand, since $F(t)$,
as a continuous periodic function, is bounded, we have trivially
$|F(t+h)-F(t)|\le c = h^0$ for some constant $c$ and all $t$ and $h$.
The estimate \eqref{eq:F-bound} thus
represents a saving of the same power of $h$, namely $h^{1-\alpha}$,
over this trivial bound. This heuristic also suggests that the exponent
$1-\alpha$ in \eqref{eq:F-bound} is best-possible. In fact, in 
Section \ref{sec:applications} we present examples for which the
exponent $1-\alpha$ is indeed optimal (see Corollaries
\ref{cor:weierstrass} and \ref{cor:riemann}).


The proof of Theorem \ref{thm:main}, given in Section \ref{sec:proof},
is based on the summation by parts formula, a standard technique in
analytic number theory (see Lemma \ref{lem:abel} below). The theorem can be generalized in a variety of directions, using
essentially the same approach.  For example, one can replace the weights
$1/n$ in \eqref{eq:F-def} by weights of the form $1/n^p$ for some $p>0$,
or even by a more general sequence of weights $\{w_n\}$ belonging to
some $l_p$-space and satisfying appropriate regularity conditions. One
can also generalize the frequencies $n$ in the exponential terms
$e^{2\pi i nt}$ in \eqref{eq:S-def} and \eqref{eq:F-def} to more general
frequencies $\phi(n)$ subject to some growth conditions.   

Here we confine ourselves to proving a generalization 
to Fourier series of the form 
\begin{equation} 
\label{eq:Fkp-def0}
\Fkp(t)=\sum_{n=1}^\infty
\frac{f(n)e^{2\pi i n^k t}}{n^p},
\end{equation}
where $k$ is an arbitrary positive integer and $p$ 
an arbitrary positive real number. Let 
\begin{equation}
\label{eq:Sk-def}
S_k(x,t)=\sum_{1\le n\le x}f(n)e^{2\pi i n^kt}
\end{equation}
be the exponential sum associated with the series \eqref{eq:Fkp-def0}.

\begin{Thm}
\label{thm:main-generalized} 
Assume that for some positive constants $C$ and $\alpha$ satisfying
\begin{equation}
\label{eq:alpha-bound}
\max(0,p-k)<\alpha<p
\end{equation}
we have
\begin{equation}
\label{eq:Sk-bound}
|S_k(x,t)|\le Cx^{\alpha}\quad (x\ge 1,\ t\in\R).
\end{equation}
Then we have
\begin{equation}
	\label{eq:Fk-bound}
	|\Fkp(t+h)-\Fkp(t)|\le C_2 h^{(p-\alpha)/k}
	\quad (t\in\R,\ h>0),
\end{equation}
where $C_2$ is the constant $C_1$ of Theorem \ref{thm:main}, 
with $C$ and
$\alpha$ replaced by $\Ct=(1+|k-p|/(\alpha +k-p))C$ and $\alphat
=(\alpha+k-p)/k$, respectively, i.e., 
\begin{align}
\label{eq:C2-formula}
C_2&=C_2(k,p,C,\alpha)=
C_1\left(\left(1+\frac{|k-p|}{\alpha + k-p}\right)C,
\frac{\alpha+k-p}{k}\right)
\\
\notag
&=\frac{6\pi\left(1+\frac{|k-p|}{\alpha + k-p}\right)C}
{\frac{\alpha + k - p}{k}\cdot
\left(1-\frac{\alpha+k-p}{k}\right)}
=\frac{6\pi C \left(\alpha + k-p+|k-p|\right)k^2}
{\left(\alpha+k-p\right)^2(p-\alpha)}.
\end{align}
Moreover, the same conclusion holds for the functions
$|\Fkp(t)|$, $\Re \Fkp(t)$, and $\Im \Fkp(t)$. 

In particular, the functions $\Fkp(t)$, $|\Fkp(t)|$, $\Re \Fkp(t)$, and
$\Im \Fkp(t)$ are H\"older continuous with exponent  $(p-\alpha)/k$.
\end{Thm}


Note that, when $k=p=1$, the functions $\Fkp(t)$ and $S_k(t)$ reduce 
to the functions $F(t)$ and $S(t)$ defined in \eqref{eq:F-def} and
\eqref{eq:S-def}, respectively, the condition \eqref{eq:alpha-bound} on
$\alpha$ reduces to $0<\alpha<1$, and the bound $\le C_2
h^{(p-\alpha)/k}$ in \eqref{eq:Fk-bound} 
of Theorem \ref{thm:main-generalized}
reduces to the bound $\le C_1 h^{1-\alpha}$ in 
\eqref{eq:F-bound} of Theorem \ref{thm:main}. Thus, Theorem
\ref{thm:main-generalized} generalizes Theorem \ref{thm:main}.  
However, in proving Theorem \ref{thm:main-generalized}, we will make use
of Theorem \ref{thm:main}.

As in the case of Theorem \ref{thm:main}, Theorem
\ref{thm:main-generalized} implies bounds on the dimensions of the
graphs $G(\Fkp)$, $G_R(\Fkp)$ and $G_I(\Fkp)$ associated with $\Fkp(t)$,
which are defined in the same way the graphs $G(F)$,  $G_R(F)$, and
$G_I(F)$.

\begin{Cor}
\label{cor:main-generalized}
Under the assumptions of Theorem \ref{thm:main-generalized} we have:
\begin{itemize}
\item[(i)]
The upper
box-counting dimension of the image $\Fkp([0,1])$ of the interval
$[0,1]$ under $\Fkp$ satisfies
\begin{align}
                \label{eq:dimGFkp}
                \overline{\dim_B}(\Fkp([0,1]))&\leq\frac{k}{p-\alpha}.
\end{align}
\item[(ii)]
The upper box-counting dimension
of the graph $G(\Fkp)$ defined in \eqref{eq:GF-def}
satisfies
        \begin{align}
                \label{eq:dimFkp}
                \overline{\dim_B}(G(\Fkp))&\leq 2-\frac{p-\alpha}{k}.
\end{align}
Moreover, the same bound holds for the
upper box-counting dimensions of the graphs
$G_R(\Fkp)$ and $G_I(\Fkp)$, defined in \eqref{eq:GRF-def} and
\eqref{eq:GIF-def}, respectively. 

\end{itemize}
\end{Cor}



Series of the form \eqref{eq:F-def} and \eqref{eq:Fkp-def0} arise in a wide
variety of contexts in number theory, analysis, and fractal geometry,
and graphs associated with such series 
often exhibit interesting fractal properties. 
A typical example is the series 
\begin{equation}
\label{eq:Fmu-def0}
F(\mu;t)=\sum_{n=1}^\infty
\frac{\mu(n)e^{2\pi i nt}}{n},
\end{equation}
where $\mu(n)$ is the M\"obius function. 
The graphs $F(\mu;[0,1])$ and $G(\mu;F)$ associated with
the M\"obius series $F(\mu;t)$ are shown in Figure \ref{fig:moebius12}.
\begin{figure}[H]
        \begin{subfigure}{.5\textwidth}
                \centering
                \includegraphics[width=.9\linewidth]{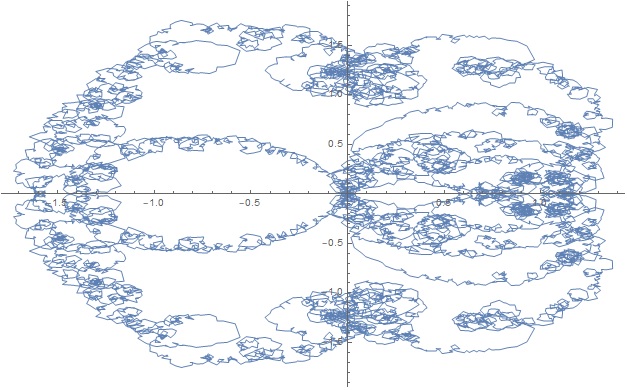}
        \end{subfigure}%
        \begin{subfigure}{.5\textwidth}
                \centering
                \includegraphics[width=.9\linewidth]{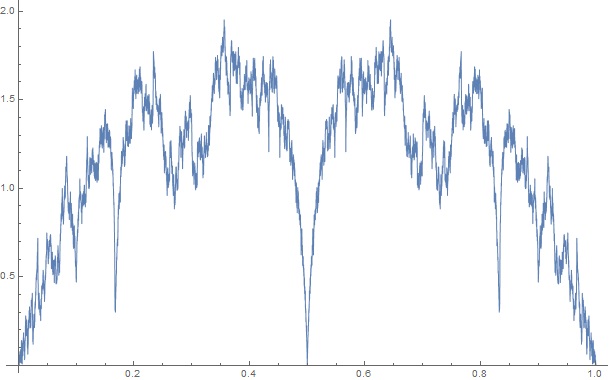}
        \end{subfigure}
        \caption{The path $F(\mu;[0,1])$ (left figure)
	and the graph $G(\mu;F)$ (right figure) associated   
	with the M\"obius series 
        $F(\mu;t)=\sum_{n=1}^\infty \mu(n)e^{2\pi i n t}/n$.}
        \label{fig:moebius12}
\end{figure}
Bohman and Fr\"oberg \cite{bohman-froberg1995} investigated the series
\eqref{eq:Fmu-def0} and some related series numerically, focusing on their  
fractal properties. In particular, these authors were the first to observe the
peculiar shape of the path $F(\mu;[0,1])$ shown on the left of Figure
\ref{fig:moebius12}. 

On the theoretical side, Bateman and Chowla \cite{bateman-chowla1963} proved
that the series $F(\mu;t)$ converges uniformly and thus represents a continuous
function of $t$.
More precise regularity properties for
the series $F(\mu;t)$ and similar series involving the M\"obius function
have been recently established by Veech \cite{veech2018} under the
assumption of the Generalized Riemann Hypothesis (GRH).
As applications of Theorems \ref{thm:main} and \ref{thm:main-generalized}
and their corollaries
we will obtain, under the same assumption,  
non-trivial bounds for H\"older exponents and dimensions of graphs
associated with the M\"obius Fourier series
$F(\mu;t)$  and its generalizations $F_{k,k}(\mu;t)$ defined in
\eqref{eq:Fkp-def0} (see Corollaries \ref{cor:Moebius} and
\ref{cor:Moebius2}).


Another class of functions to which our results
can be applied are the 
Riemann and Weierstrass functions, defined by 
\begin{equation}
	\label{eq:Riemann-def0}
\sum_{n=1}^\infty \frac{\sin(\pi n^2t)}{n^2}
\end{equation}
and 
\begin{equation}
\label{eq:Weierstrass-def0}
\sum_{k=1}^\infty a^k \cos(2\pi b^k t),
\end{equation}
respectively,
where $a$ and $b$ are real numbers satisfying $0<a<1$ and $b>1/a$.
These functions are classical examples of 
continuous functions that are almost everywhere non-differentiable, and they  
have been extensively studied in the
literature, both from a regularity point of view and from a
geometric/fractal point of view; see, for example,
\cite{baranski2014},
\cite{baranski2015},
\cite{duistermaat1991}
\cite{ecei2020},
\cite{hardy1916}, 
\cite{hunt1998}, 
\cite{jaffard1996},
\cite{kaplan1984},
\cite{shen2018}. For additional related work see \cite{JAextra1}, \cite{JAextra2}, \cite{JAextra3}, \cite{JAextra4}.

In Section \ref{sec:applications} we will see that the Weierstrass and
Riemann functions can be viewed as special cases of the Fourier series
\eqref{eq:F-def} and \eqref{eq:Fkp-def0}, respectively.  By applying
Theorems \ref{thm:main} and \ref{thm:main-generalized} 
we will obtain bounds on the
H\"older exponent that turn out to be the best-possible \emph{uniform}
bounds of this type; see Corollaries \ref{cor:weierstrass} and
\ref{cor:riemann}.


The Weierstrass and Riemann functions can be generalized in a variety of
ways.  For example, \cite{chamizo-cordoba1999},
\cite{chamizo-ubis2007}, \cite{chamizo-petrykiewicz-ruiz2017}
investigated functions of the form
$\sum_{n=1}^\infty e^{2\pi i n^k}/{n^p}$ for $k,p\geq2$,
which can be regarded as 
complex generalizations of the Riemann function, and more generally
functions of the form \eqref{eq:Fkp-def0} 
under certain regularity  
assumptions on the coefficients $f(n)$ (such as monotonicity or
growth constraints).  We note
that our results apply to a much broader class of coefficient sequences,
though the conclusions are not as precise as those obtained in the cited
references.



\subsection*{Outline of the paper.}
In the hope that the results of this paper will be of interest to
researchers in both analysis and number theory, we tried to keep the
exposition broadly accessible, recalling definitions and results that
may be only known to specialists in the relevant area.

In Section \ref{sec:background} we recall the definitions of H\"older
continuity, H\"older exponent,  and box-counting dimension, and we cite 
a key result relating the latter two quantities. We then use this
result to deduce the corollaries from Theorems \ref{thm:main} and
\ref{thm:main-generalized}.  Section \ref{sec:proof} contains the proof
of Theorem \ref{thm:main}, while Section \ref{sec:generalization}
contains the proof of  Theorem \ref{thm:main-generalized}.

In Section \ref{sec:applications} we present the applications of these
results mentioned above to Riemann and Weierstrass type functions and to
Fourier series associated with the M\"obius function.  In the final
section,  Section \ref{sec:conclusion}, we discuss possible extensions
and generalizations of our results and open problems suggested by these
results.


\section{Background on H\"older continuity and dimensions and proof of
the corollaries}
\label{sec:background} 

We begin by defining H\"older continuity, H\"older exponents, and pointwise
H\"older exponents. 

\begin{Def}[{\cite[p.~8]{falconer2014} and \cite{jaffard1996}}]
\label{def:holder}
Let $I\subseteq \R$ be a non-trivial compact interval and let $f:I\to \R^{d}$.
\nopagebreak
\begin{itemize}
\item[(i)] The function $f$ is called
\textbf{H\"older continuous} if there are constants $C>0$ and $\eta\in(0,1)$
such that 
\begin{equation}
\label{eq:Holder-def}
|f(x)-f(y)|\leq C|x-y|^\eta
\end{equation}
holds for all $x,y\in I$.  
We call $\eta$ a \textbf{H\"older exponent} of $f$.

\item[(ii)] Let $x_0\in I$. The function $f$ is called 
\textbf{locally H\"older continuous at $x_0$}
if there are constants $C>0$ and $\eta>0$ 
and a polynomial $P$ of degree less than $\eta$ such that 
\begin{equation}
\label{eq:local-Holder-def}
|f(x)-P(x-x_0)|\leq C|x-x_0|^\eta
\end{equation}
holds for all $x$ in some open neighborhood of $x_0$.  
The exponent $\eta$ 
is called a \textbf{local H\"older exponent} of $f$ at $x_0$.
The supremum of all $\eta$ for which \eqref{eq:local-Holder-def} holds is 
called the \textbf{pointwise H\"older exponent} of $f$ at $x_0$ 
and is denoted by $\eta(x_0)$.  
\end{itemize}
\end{Def}

We note that, if $0<\eta<1$, \eqref{eq:local-Holder-def} reduces to 
\begin{equation}
\label{eq:local-Holder-def2}
|f(x)-f(x_0)|\leq C|x-x_0|^\eta
\end{equation}
as the polynomial $P(x)$ must be of degree $0$ and hence be a constant. 


We next recall the definition of the \emph{box-counting dimension},
which is  one of several standard concepts of dimensions in fractal
geometry.

\begin{Def}[{\cite[p.~6]{fraser2021}}]
        \label{def:box-dim}
	Let $E$ be a bounded subset of $\R^d$. For $r>0$, denote by
	$N(E,r)$ the minimal number of sets of diameter at most $r$
	needed to cover $E$. The \textbf{upper box-counting dimension}
	of $E$ is defined as
      \begin{equation}
      \overline{\dim_B} E= \limsup_{r\to 0} \frac{\log
      N(E,r)}{\log(1/r)}.
      \end{equation} 
      Similarly, the \textbf{lower
      box-counting dimension} of $E$ is defined as 
      \begin{equation}
      \underline{\dim_B} E = \liminf_{r\to 0} \frac{\log
      N(E,r)}{\log(1/r)}. 
      \end{equation}
      When the upper and lower
      box-counting dimensions coincide, we call the common value the
      \textbf{box-counting dimension} of $E$ and denote it by $\dim_B
      E$. 
\end{Def}

The following proposition relates the H\"older exponent of a function to
the (upper) box-counting dimension of sets associated with this function. 

\begin{Prop}
[{\cite[p.~49]{fraser2021}} and {\cite[Cor.~11.2]{falconer2014}}]
\label{prop:dimdist}
Let $I\subseteq \R$ be a nontrivial compact interval, and let $f:I\to
\R^d$
be a H\"older continuous function with exponent $\eta>0$.  Then we have:
\begin{itemize}
\item[(i)]
The upper box-counting dimension
of $f(I)\subseteq R^d$ satisfies
\begin{equation}
\label{eq:dimineq2}
\overline{\dim_B} f(I)\leq \frac{1}{\eta}.
\end{equation}
\item[(ii)] In the case when $d=1$,  
the upper box-counting dimension of the graph
$G(f)=\{ (x,f(x))\in\R^2: x\in I\}$ satisfies
\begin{equation}
\label{eq:dimineq3}
\overline{\dim_B} G(f) \leq 2-\eta.
\end{equation}
\end{itemize}
\end{Prop}

\begin{proof}[Proof
of Corollaries \ref{cor:main} and \ref{cor:main-generalized}]
Corollary
\ref{cor:main} is the special case $k=p=1$ of Corollary
\ref{cor:main-generalized}, so it suffices to prove the latter result.

Assume that $S_k(t)$ satisfies the assumptions of Corollary
\ref{cor:main-generalized}.  By Theorem \ref{thm:main-generalized} it  
follows that the functions $\Fkp(t)$,
$|\Fkp(t)|$, $\Re \Fkp(t)$, and $\Im \Fkp(t)$, 
are H\"older continuous with exponent $\eta=(p-\alpha)/k$.
Applying Proposition \ref{prop:dimdist} then shows 
that the box-counting dimension 
of  $\Fkp([0,1])$ is bounded above by $1/\eta=k/(p-\alpha)$,
while the box-counting dimensions of the graphs $G(\Fkp)$, $G(|\Fkp|)$,
$G(\Re\Fkp)$, and $G(\Im\Fkp)$ are bounded above
by  $2-\eta=2-(p-\alpha)/k$.  These bounds are exactly the dimension bounds 
\eqref{eq:dimGFkp} and \eqref{eq:dimFkp} of
Corollary \ref{cor:main-generalized}.
\end{proof}


\section{Proof of Theorem \ref{thm:main}}
\label{sec:proof}

The proof depends in a crucial way on the summation by parts formula (Abel's identity),
a standard technique in analytic number theory.  For the convenience of
the reader, we recall this formula in the following lemma.


\begin{Lem}
	\label{lem:abel}
	Let $a:\N\to\C$ be an arithmetic function, and let $A(x)=\sum_{1\le n\le x}
	a(n)$, with the convention that $A(x)=0$ if $x<1$. 
	\begin{itemize}
		\item[(i)] Let $0<y<x$ be real numbers and
		assume that $\phi(u)$ is defined on the closed interval $[y,x]$ and has
		a continuous derivative on this interval.  Then
		\begin{equation}
			\label{eq:abel}
			\sum_{y<n\le x}a(n)\phi(n)=A(x)\phi(x)-A(y)\phi(y)-\int_y^x
			A(u)\phi'(u)du.
		\end{equation}
		\item[(ii)] Let $x\ge 1$ and assume that
		$\phi(u)$ is defined and has a continuous derivative on the interval
		$(0,x]$. Then
		\begin{equation}
			\label{eq:abel2}
			\sum_{1\le n\le x}a(n)\phi(n)=A(x)\phi(x)-\int_1^x A(u)\phi'(u)du.
		\end{equation}
	\end{itemize}
\end{Lem}

\begin{proof}
	Assertion (i) of the lemma is exactly Lemma 4.2
	of \cite{apostol1976}. Assertion (ii) follows from (i) on letting  
	$y=1/2$ in \eqref{eq:abel} and noting that $A(u)=0$ for $u<1$.
\end{proof}

For the remainder of this section, we assume that $f:\N\to\C$ is an
arithmetic function, we let $S(x,t)$ be the associated exponential sum
defined in \eqref{eq:S-def}, and we assume that $S(x,t)$ satisfies the
hypothesis \eqref{eq:S-bound} of Theorem \ref{thm:main} with some
constants $C$ and $\alpha$ with $0<\alpha<1$.



Given a positive integer $N$, let $F_N(t)$ denote the $N$th partial sum of $F(t)$, i.e.,  
\begin{equation}
	\label{eq:FN-def}
	F_N(t)=\sum_{1\leq n\leq N}\frac{f(n)e^{2\pi i nt}}{n}.
\end{equation}


\begin{Lem}
	\label{lem:cauchy-estimate}
	Uniformly in integers $M>N\ge 1$ and $t\in\R$ we have
	\begin{equation}
		\label{eq:cauchy-estimate}
		|F_M(t)-F_N(t)|\le C_{11} N^{\alpha-1},
	\end{equation}
where
\begin{equation}
\label{eq:C11-def}
C_{11}=C_{11}(C,\alpha)=\frac{2C}{1-\alpha}.
\end{equation}
\end{Lem}

\begin{proof}
	Applying Lemma \ref{lem:abel}(i) with $y=N$, $x=M$, 
	$a(n)=f(n)e^{2\pi i n t}$ (so that
	$A(u)=S(u,t)$) and $\phi(u)=1/u$ 
	along with the estimate \eqref{eq:S-bound}, we obtain
	\begin{align*}
		\left|F_M(t)-F_N(t)\right|&=
		\left|\sum_{N<n\le M} \frac{f(n)e^{2\pi i nt}}{n}\right|
		\\
		&=\left|\frac{S(M,t)}{M}-\frac{S(N,t)}{N}
		+\int_N^M \frac{S(u,t)}{u^2}du\right|
		\\
		&\le C
		\left(M^{\alpha-1}+N^{\alpha-1}+\int_N^M
		u^{\alpha-2}du\right)
		\\
		&= C\left(M^{\alpha-1} + N^{\alpha-1}
		+\frac{1}{1-\alpha}\left(N^{\alpha-1}-M^{\alpha-1}\right)\right)
		\\
		&\le C\left(N^{\alpha-1} 
		+\frac{N^{\alpha-1}}{1-\alpha}
		\right)
		\\
		&\le \frac{2C}{1-\alpha} N^{\alpha-1} =C_{11} N^{\alpha-1},
	\end{align*}
	as claimed.	
\end{proof}


\begin{Lem}
	\label{lem:tail-estimate}
	The series $F(t)$ converges uniformly in $t$, and we have 
	\begin{equation}
		\label{eq:tail-estimate}
		F(t)-F_N(t)\le C_{11} N^{\alpha-1}.
	\end{equation}
\end{Lem}

\begin{proof}
	Both assertions follow from Lemma \ref{lem:cauchy-estimate} on letting
	$M\to\infty$.
\end{proof}


\begin{Lem}
	\label{lem:FN-diff}
	Uniformly for $t\in \R$, $h>0$, and any positive integer $N$ 
	we have
	\begin{equation}
		\label{eq:FN-diff}
		|F_N(t+h)-F_N(t)|\le C_{12}hN^{\alpha},
	\end{equation}
where
\begin{equation}
\label{eq:C12-def}
C_{12}=C_{12}(C,\alpha)=\frac{6\pi C}{\alpha}.
\end{equation}

\end{Lem}

\begin{proof}
	We have
	\begin{align}
		\label{eq:lem2a}
		F_N(t+h)-F_N(t)&
		=\sum_{n=1}^N\frac{f(n)\left(e^{2\pi i n(t+h)}-e^{2\pi i nt}\right)}{n}
		=\sum_{n=1}^N f(n)e^{2\pi i nt}\phi(n),
	\end{align}
	where
	\begin{equation}
		\label{eq:phi-def}
		\phi(x)=\frac{e^{2\pi ihx}-1}{x}.
	\end{equation}
	Applying Lemma \ref{lem:abel}(ii) with $a(n)=f(n)e^{2\pi i n t}$ 
	and $\phi(x)$ defined by \eqref{eq:phi-def}, 
	the sum on the right of \eqref{eq:lem2a} becomes
	\begin{align}
		\label{eq:lem2b}
		\sum_{n=1}^N f(n)e^{2\pi i nt}\phi(n) &=
		S(N,t)\phi(N)-\int_1^NS(x,t)\phi'(x)dx.
	\end{align}
	Using the elementary inequality $|e^{iu}-1|\le u$ for all $u\in \R$, 
	we see that the functions $\phi(x)$
	and $\phi'(x)$ in \eqref{eq:lem2b} satisfy
	\begin{align}
		\label{eq:phi-bound}
		|\phi(x)|&=\frac{|e^{2\pi i hx}-1|}{x}\le
		\frac{|2\pi hx|}{x}= 2 \pi h,
		\\
		\label{eq:phiprime-bound}
		|\phi'(x)|&\le
		\frac{|e^{2\pi i hx}-1|}{x^2}+\frac{|2\pi i he^{2\pi i
				hx}|}{x}
		\le 2\pi \frac{h}{x} + 2\pi \frac{h}{x} = 4\pi
		\frac{h}{x}.
	\end{align}
	Substituting the bounds \eqref{eq:phi-bound}
	and \eqref{eq:phiprime-bound} along with our
	assumption \eqref{eq:S-bound}
	into the right-hand side of \eqref{eq:lem2b}, we obtain 
	\begin{align}
		\label{eq:lem2c}
		\left|\sum_{n=1}^N f(n)e^{2\pi i nt}\phi(n)\right|
		&\le 2\pi h C N^\alpha  + 4\pi h C 
		\int_1^N x^{\alpha-1}dx
		\\
		\notag
		&\le ChN^\alpha\left(2\pi + \frac{4\pi}{\alpha}\right)
		\\
		\notag
		&\le Ch N^\alpha \frac{6\pi}{\alpha} 
		= C_{12} h N^{\alpha}.
	\end{align}
	Combining this with \eqref{eq:lem2a} yields the
	desired estimate \eqref{eq:FN-diff}.
\end{proof}


\begin{proof}[Proof of Theorem \ref{thm:main}]
For the proof of the bound \eqref{eq:F-bound},
let $t\in\R$ and $h>0$ be given. Since $F(t)$ is a periodic function
with period $1$, we may assume that
\begin{equation}
\label{eq:h-bound}
0<h\le 1.
\end{equation}

Let $N$ be a positive integer. Then 
	\begin{align}
		\label{eq:F-diff}
		\tiny|F(t+h)-F(t)|&= 
		\Bigl|\left(F(t+h)-F_N(t+h)\right)
		-\left(F(t)-F_N(t)\right)+
		\left(F_N(t+h)-F_N(t)\right)\Bigr|
		\\
		\notag
		&\le \left|F(t+h)-F_N(t+h)\right| 
		+ \left|F(t)-F_N(t)\right| + 
		\left|F_N(t+h)-F_N(t)\right|. 
	\end{align}
	By Lemma \ref{lem:tail-estimate} the first two terms on the right 
	of \eqref{eq:F-diff}
	are bounded by $\le C_{11} N^{\alpha-1}$ each, 
	and by Lemma \ref{lem:FN-diff} the third term is bounded by 
	$\le C_{12} h N^{\alpha}$. 
	It follows that 
	\begin{equation}
		\label{eq:F-diff2}
		|F(t+h)-F(t)|\le 2C_{11} N^{\alpha-1} + C_{12} h N^{\alpha}
	\end{equation}
	for any positive integer $N$. To (approximately)
	optimize this bound, we choose $N$ as 
	\begin{equation}
		\label{eq:N-def}
		N=N_h=\left\lfloor 1/h\right\rfloor,
	\end{equation}
	where $\left\lfloor \cdot\right\rfloor$ is the floor function.
	In view of our assumption \eqref{eq:h-bound},
	$N$ is a positive integer satisfying $1/(2h)\le N\le 1/h$.
	Therefore we have 
	\begin{align*}
	2C_{11}	N^{\alpha-1}+ C_{12}hN^\alpha
	&\le 2C_{11}	(2h)^{1-\alpha}+ C_{12} h^{1-\alpha}
	\\
	&\le (4C_{11}+C_{12})h^{1-\alpha}
	\\
	&= C\left(\frac{8}{1-\alpha}+
	\frac{6\pi}{\alpha}\right) h^{1-\alpha}
	\\
	&\le \frac{6\pi C}{\alpha(1-\alpha)} h^{1-\alpha} = C_1
	h^{1-\alpha}.
	\end{align*}
	Combining this with \eqref{eq:F-diff2} yields 
	the inequality \eqref{eq:F-bound} for $|F(t+h)-F(t)|$.
	
	In view of the elementary inequalities
	\begin{align*}
		||F(t+h)|-|F(t)||&\le |F(t+h)-F(t)|,
		\\
		|\Re F(t+h)-\Re F(t)|&\le |F(t+h)-F(t)|,
		\\
		|\Im F(t+h)-\Im F(t)|&\le |F(t+h)-F(t)|
	\end{align*}
	the same
	conclusion holds for the functions $|F(t)|$, $\Re F(t)$, and $\Im
	F(t)$.   
	
	This completes the proof of Theorem \ref{thm:main}.
\end{proof}



\section{Proof of Theorem \ref{thm:main-generalized}} 
\label{sec:generalization}

Let $k$ be a positive integer and $p$ a positive real number.
Let  $f:\N\to\C$
be an arithmetic function and let $S_k(x,t)$ and $\Fkp(t)$ denote  
the associated exponential sums and Fourier series
defined by (see \eqref{eq:Sk-def} and \eqref{eq:Fkp-def0})
\begin{align*}
S_k(x,t)&=\sum_{1\le n\le x}f(n)e^{2\pi i n^kt},
\\
\Fkp(t)&=\sum_{n=1}^\infty \frac{f(n)e^{2\pi i n^kt}}{n^p}.
\end{align*}

Define an arithmetic function $\ft$ by
\begin{equation}
	\label{eq:ft-def}
	\ft(n)=\begin{cases}
		m^{k-p}f(m) &\text{if $n=m^k$ for some $m\in\N$,}
		\\
		0&\text{otherwise,}
	\end{cases}
\end{equation}
and let $\St(x,t)$ and $\Ft(t)$ be defined as in \eqref{eq:S-def} and
\eqref{eq:F-def}, but with respect to the arithmetic function $\ft(n)$.
Then
\begin{align}
	\label{eq:St-S}
	\St(x,t)&=\sum_{1\le n\le x} \ft(n)e^{2\pi i nt}
	=\sum_{m\le x^{1/k}}m^{k-p}f(m)e^{2\pi i m^k t},
	\\
	\label{eq:Ft-F}
	\Ft(t)&
	=\sum_{n=1}^\infty \frac{\ft(n)e^{2\pi i nt}}{n}
	=\sum_{m=1}^\infty \frac{m^{k-p}f(m)e^{2\pi i m^kt}}{m^k}
	=\Fkp(t).
\end{align}
Applying Lemma \ref{lem:abel} with $\phi(u)=u^{k-p}$ and
$a(m)=f(m)e^{2\pi i m^k t}$ 
(so that $A(u)= \sum_{m\le u}f(m)e^{2\pi i m^k t}\allowbreak
= S_k (u,t)$),  
we obtain
\begin{align*}
\left|\sum_{m\le x^{1/k}}m^{k-p}f(m)e^{2\pi i m^k t}\right|
&=\left|S_k(x^{1/k},t) x^{(k-p)/k}-\int_1^{x^{1/k}}
S_k(u,t)(k-p)u^{k-p-1}du\right|
\\
\notag
&\le C 
x^{\alpha/k} x^{(k-p)/k} +
C|k-p|\int_{1}^{x^{1/k}} u^{\alpha}u^{k-p-1}du.
\end{align*}
Note that, by \eqref{eq:alpha-bound}, $0<\alpha+k-p<k$. Hence,
\begin{align}
	\label{eq:St-S2}
	\left|\sum_{m\le x^{1/k}}m^{k-p}f(m)e^{2\pi i m^k t}\right|
	&\leq C x^{(\alpha+k-p)/k} +C|k-p|\frac{x^{(\alpha+k-p)/k}-1}{\alpha+k-p}
	\\
	\notag
	&\le \Ct  x^{(\alpha+k-p)/k},
\end{align}
where 
\begin{equation}
\label{eq:C1t-formula}
\Ct =\left(1+\frac{|k-p|}{\alpha+k-p}\right)C.
\end{equation}

Combining \eqref{eq:St-S2} with \eqref{eq:St-S} yields  
\begin{equation}
\label{eq:St-bound}
|\St(x,t)| \leq \Ct x^{(\alpha+k-p)/k}=\Ct x^{\alphat},
\end{equation}
where 
\begin{equation}
\label{eq:alphat-formula}
\alphat=\frac{\alpha+k-p}{k}.
\end{equation}
Our assumption
\eqref{eq:alpha-bound} ensures that
the exponent $\alphat$ defined in \eqref{eq:alphat-formula}
satisfies $0<\alphat<1$. We can therefore apply 
Theorem \ref{thm:main} to conclude that
\begin{align}
	\label{eq:Ft-bound}
	|\Ft(t+h)-\Ft(t)|&\le \Conet h^{1-\alphat}
	\quad (t\in\R,\, h>0),
\end{align}
where 
\begin{align*}
\Conet &=C_1(\Ct,\alphat)
= C_1\left(\left(1+\frac{|k-p|}{\alpha + k-p}\right)C,
\frac{\alpha+k-p}{k}\right)
=C_2(k,p,C,\alpha)
\end{align*}
is the constant in Theorem \ref{thm:main-generalized}.
In view of \eqref{eq:Ft-F} this proves the 
bound \eqref{eq:Fk-bound} of Theorem \ref{thm:main-generalized}
for $\Fkp(t)$, which in turn implies analogous bounds (with the same
constant) for the functions $|\Fkp(t)|$,
$\Re \Fkp(t)$, and $\Im \Fkp(t)$.

This completes the proof of Theorem \ref{thm:main-generalized}.
\qed


\section{Applications}\label{sec:applications}


\subsection{Weierstrass type functions}

The classical Weierstrass functions are defined by  
\begin{equation}
\label{eq:Weierstrass-def}
\sum_{k=1}^\infty a^k \cos(2\pi b^k t),
\end{equation}
where $a$ and $b$ are positive real numbers satisfying $1/b<a<1$. These
functions have been introduced more than a century ago 
by Weierstrass and Hardy (see \cite{hardy1916})  as examples of
continuous, but nowhere differentiable functions, and they 
have been extensively studied in the literature.
In particular, it is now known that the function
\eqref{eq:Weierstrass-def} is H\"older continuous with 
exponent $-\log_ba$, that this exponent is both globally and locally
optimal \cite{baranski2015}, and that the Hausdorff \cite{shen2018} and box-counting dimensions \cite{hardy1916}
of the graph of this function over the interval $[0,1]$ are both 
equal to $2+\log_ba$.

Here we consider generalizations  
of the Weierstrass function of the form
\begin{align}
\label{eq:Wab-def}
W_{a,b}(f;t)&=\sum_{k=1}^\infty f(k) a^k e^{2\pi i b^k t},
\end{align}
with arbitrary bounded coefficients $f(k)$.  
From Theorem \ref{thm:main}
we will derive the following result, which shows that, for integer
values of $b$,
this much more
general class of functions satisfies the same  H\"older and dimension
bounds as the classical Weierstrass function \eqref{eq:Weierstrass-def}.


\begin{Cor}
\label{cor:weierstrass}
Let $a$ be a real number satisfying $0<a<1$, let $b$ be an integer
satisfying $b>1/a$, and let $f:\N\to\C$ be an arbitrary bounded function.
Then we have:
\begin{itemize}
\item[(i)] The function $W_{a,b}(f;t)$ 
is H\"older continuous with exponent $-\log_b a$, and the same holds
for the real-valued functions
$|W_{a,b}(f;t)|$, $\Re W_{a,b}(f;t)$, and $\Im
W_{a,b}(f;t)$.
\item[(ii)]
The upper box-counting dimensions of the graphs of the  
functions $|W_{a,b}(f;t)|$, $\Re W_{a,b}(f;t)$, and $\Im W_{a,b}(f;t)$
over the interval $[0,1]$ are bounded by $\le 2+\log_b a$.
\end{itemize}
\end{Cor}

\begin{Remark}
The restriction of the parameter $b$ in the corollary 
to integer values is a purely technical one as it allows us to derive
the result directly from Theorem \ref{thm:main}.  
By adapting the proof of Theorem \ref{thm:main}, one can show 
that the corollary remains valid without this restriction. 
\end{Remark}

\begin{proof}
Let $f:\N\to\C$ be a bounded function. 
By rescaling $f$ if necessary we may assume, without loss of generality, 
that $|f(n)|\le 1$ for all $n\in\N$.  Define an arithmetic function $\ft$ by
\begin{equation}
\label{eq:fabt-def}
\ft(n)=\begin{cases}
(ab)^kf(k) &\text{if $n=b^k$ for some $k\in\N$,}
			\\
0&\text{otherwise,}
\end{cases}
\end{equation}
and let $\St(x,t)$ and $\Ft(t)$ be defined as in \eqref{eq:S-def} and
\eqref{eq:F-def}, but with respect to the arithmetic function $\ft(n)$.
Setting 
\begin{equation}
\label{eq:kx-def}
k_x=\lfloor \log_bx\rfloor,
\end{equation}
we then have 
\begin{align}
\label{eq:Sabt-S}
\St(x,t)&=\sum_{1\le n\le x} \ft(n)e^{2\pi i nt}
		=\sum_{1\le k\le k_x}(ab)^{k}f(k)e^{2\pi i b^k t},
		\\
		\label{eq:Fabt-F}
		\Ft(t)&
		=\sum_{n=1}^\infty \frac{\ft(n)e^{2\pi i nt}}{n}
		=\sum_{k=1}^\infty f(k)a^k e^{2\pi i b^kt}
=W_{a,b}(f;t).	
\end{align}
Since, by assumption,  $b>1/a$ and $|f|\le 1$, we have
\begin{align}
\label{eq:Sabt-S2}
\left|\sum_{1\le k\le k_x}(ab)^{k}f(k)e^{2\pi i b^k t}\right|
&\le \left|\sum_{1\le k\le k_x}(ab)^{k}\right|
\le\frac{(ab)^{k_x+1}}{ab-1}
\\
&\le \frac{(ab)^{\log_b x+1}}{ab-1}
\le\frac{ab}{ab-1}x^{1+\log_ba}.
\notag
\end{align}
Combining \eqref{eq:Sabt-S2} with \eqref{eq:Sabt-S}, it follows that the
function $\ft$ satisfies the condition \eqref{eq:S-bound} of Theorem
\ref{thm:main} with exponent $\alpha=1+\log_ba=\log(ab)/\log b$ and
constant $C=C_{a,b}=ab/(ab-1)$. 
Moreover, our assumptions $0<a<1$ and  $ab>1$ ensure that $\alpha$
satisfies the condition $0<\alpha<1$ of  Theorem \ref{thm:main}.
The theorem therefore implies 
that the Fourier series $\Ft(t)$ is H\"older continuous with
exponent $1-\alpha=-\log_ba$, and that the same holds for the 
functions $|\Ft(t)|$, $\Re \Ft(t)$, and $\Im \Ft(t)$.  Since, by
\eqref{eq:Fabt-F}, $\Ft(t)=W_{a,b}(t)$, this proves the claim of part
(i) of the corollary.

Part (ii) follows from Corollary \ref{cor:main}.
\end{proof}

\begin{Remark}
In the case when $f\equiv1$, the function $\Re W_{a,b}(t)$ reduces to
the classical Weierstrass function \eqref{eq:Weierstrass-def}. In view
of the remarks at the beginning of the section, the bounds $-\log_ba$
and $2+\log_ba$ for the  H\"older exponent and dimensions provided by
the corollary are best-possible in this case.   This example shows that
the bounds provided by Theorem \ref{thm:main} 
on the H\"older exponent and dimensions are optimal. 
\end{Remark}


\subsection{Riemann type functions}
The Riemann function
\begin{equation}
	\label{eq:Riemann-def}
\sum_{n=1}^\infty \frac{\sin(\pi n^2t)}{n^2}
\end{equation}
is another classical example of a function that is continuous, but
non-differentiable almost everywhere. 
In contrast to the Weierstrass function \eqref{eq:Weierstrass-def},
which has the same pointwise 
H\"older exponent, $\eta(t) = -\log_ba$, at each point $t$ (see
\cite{hardy1916}), 
the pointwise H\"older exponent $\eta(t)$ of the Riemann function 
is strongly dependent on the diophantine approximation properties of
$t$. More precisely, the set $\{\eta(t): t\in[0,1]\}$ 
of pointwise H\"older exponents consists of 
the closed interval $[1/2,3/4]$ along with the single point $1$;
see, for example, \cite{duistermaat1991} and \cite{jaffard1996}. 
It follows that $1/2$ is the best-possible \emph{uniform} H\"older exponent
for the Riemann function.

As an application of Theorem \ref{thm:main-generalized} we show  in
the following corollary
that the same H\"older exponent, $1/2$, holds for a class of generalized
Riemann functions defined by 
\begin{equation}
\label{eq:Rf-def}
R(f;t)=\sum_{n=1}^\infty \frac{f(n)e^{2\pi i n^2 t}}{n^2},
\end{equation}
where the coefficients $f(n)$ are bounded, but otherwise arbitrary.

\begin{Cor}
\label{cor:riemann}
Let $f:\N\to\C$ be an arbitrary bounded arithmetic function.
Then we have:
\begin{itemize}
\item[(i)] The function $R(f;t)$ 
is H\"older continuous with exponent $1/2$, and the same holds
for the real-valued functions
$|R(f;t)|$, $\Re R(f;t)$, and $\Im R(f;t)$.
\item[(ii)]
The upper box-counting dimensions of the graphs of the functions 
$|R(f;t)|$, $\Re R(f;t)$, and $\Im R(f;t)$
over the interval $[0,1]$ are bounded by $\le 3/2$.
\end{itemize}
\end{Cor}

\begin{proof}
As in the proof of Corollary \ref{cor:weierstrass},
we may assume without loss of generality 
that $f:\N\to\C$ is bounded by $1$.
Note that the series $R(f;t)$ defined in \eqref{eq:Rf-def}
is exactly the series $\Fkp(f;t)$
of Theorem \ref{thm:main-generalized} when $p=k=2$.
Moreover, by our assumption that $|f|\le 1$ we have 
\begin{align*}
|S_2(x,t)|&=\left|\sum_{1\le n\le x} f(n)e^{2\pi i n^2 t}\right|
\le \sum_{1\le n\le x} |f(n)|\le  x.
\end{align*}
Thus the assumption
\eqref{eq:Sk-bound} of Theorem
\ref{thm:main-generalized} holds with constants $C=1$ and $\alpha=1$.
Moreover, the inequality \eqref{eq:alpha-bound} is satisfied when 
$\alpha=1$ and $k=p=2$. 
The theorem therefore implies that $R(f;t)$ is H\"older continuous with
exponent $(p-\alpha)/k=(2-1)/2=1/2$. This proves part (i) of the
corollary.  The assertions of part (ii) follow from Corollary
\ref{cor:main-generalized}. 
\end{proof}


\subsection{Fourier series associated with the M\"obius and Liouville
	functions}

The M\"obius function $\mu$ is the arithmetic function defined by
$\mu(1)=1$; $\mu(n)=(-1)^r$ if $n$ is the product of $r$ \emph{distinct}
primes; and $\mu(n)=0$ otherwise. The closely related Liouville function
$\lambda$ is defined by $\lambda(1)=1$ and $\lambda(n)=(-1)^r$ if $n$ is
the product of $r$ \emph{not necessarily distinct} primes.

Davenport \cite{davenport1937} showed 
that the exponential sums $S(\mu;x,t)$ associated with the M\"obius
function satisfy
\begin{equation}
\label{eq:S-mu-bound-davenport}
	S(\mu;x,t)
	=\sum_{1\le n\le x} \mu(n) e^{2\pi i nt}\ll_A x(\log x)^{-A},
\end{equation}
uniformly in $t$, for any fixed constant $A$. Here 
the \emph{Vinogradov notation}  ``$f(x)\ll g(x)$'' 
means that there exist constants $C$ and $x_0$ such that $|f(x)|\le C
|g(x)|$ holds for all $x\ge x_0$ and the subscript $A$ in $\ll_A$
indicates that these constants may depend on $A$.

Using \eqref{eq:S-mu-bound-davenport} along with an elementary argument
that relates estimates for $S(\lambda;x,t)$ to estimates of the same
quality for
$S(\mu;x,t)$, Bateman and Chowla \cite{bateman-chowla1963} showed that
the Fourier series associated with $\mu$ and $\lambda$, i.e., the
functions
\begin{equation}
	\label{eq:Moebius-Liouville-series}
		F(\mu; t)=\sum_{n=1}^\infty\frac{\mu(n)e^{2\pi i n t}}{n},
		\quad
	F(\lambda; t)=\sum_{n=1}^\infty\frac{\lambda(n)e^{2\pi i n
	t}}{n},
\end{equation}
converge uniformly and hence represent continuous functions of $t$.

In light of the Bateman-Chowla result it is natural to ask whether these
series are H\"older continuous.
Theorem \ref{thm:main} yields such a result provided a stronger form 
of the estimate \eqref{eq:S-mu-bound-davenport}, with a saving of a power of
$x$, is available.   Unconditionally, no such estimate is known to date.
However, Baker and Harman \cite{baker-harman1991} showed that
under the assumption of the Generalized Riemann Hypothesis one has 
\begin{equation}
	\label{eq:S-mu-bound}
	\sum_{1\le n\le x} \mu(n) e^{2\pi i nt}\ll_\epsilon
	x^{3/4+\epsilon},
	\quad
	\sum_{1\le n\le x} \lambda(n) e^{2\pi i nt}
	\ll_\epsilon  x^{3/4+\epsilon},
\end{equation}
for any fixed $\epsilon>0$ and uniformly in $t$.

It is conjectured (see, e.g., (6) in \cite{balog-perelli1998}) that the
exponent $3/4$ in these estimates can be replaced by $1/2$, i.e.,
that, for any fixed $\epsilon>0$, we have uniformly in $t$
\begin{equation}
	\label{eq:S-mu-bound-conj}
	\sum_{1\le n\le x} \mu(n) e^{2\pi i nt}
	\ll_\epsilon x^{1/2+\epsilon},
	\quad
	\sum_{1\le n\le x} \lambda(n) e^{2\pi i nt}
	\ll_\epsilon
	x^{1/2+\epsilon}.
\end{equation}
In view of Parseval's identity the exponent $1/2$ here cannot be
further improved.

By Theorem \ref{thm:main} and Corollary \ref{cor:main}
the estimates \eqref{eq:S-mu-bound} and
\eqref{eq:S-mu-bound-conj} translate into bounds for the H\"older
exponent and dimensions of the Fourier series $F(\mu;t)$ and
$F(\lambda;t)$:


\begin{Cor}
\label{cor:Moebius}
Assume the Generalized Riemann Hypothesis holds. Then we have:
\begin{itemize}
\item[(i)] The functions
$F(\mu;t)$ and $F(\lambda;t)$ 
are H\"older continuous with exponent $1/4-\epsilon$, for any
$\epsilon>0$, and the same holds  
for the absolute values and the real and imaginary parts of these
functions. 
\item[(ii)]
The upper box-counting dimensions of the graphs of the
functions $|F(\mu;t)|$, $\Re F(\mu;t)$, and $\Im F(\mu;t)$ 
over the interval $[0,1]$ are bounded by $\le 7/4$, and the same 
holds for the dimensions of the graphs associated with 
$F(\lambda;t)$. 
\end{itemize}
Moreover, under conjectures \eqref{eq:S-mu-bound-conj} 
the bounds $1/4-\epsilon$ and
$7/4$ in (i) and (ii) can be replaced by $1/2-\epsilon$ and
$3/2$, respectively.
\end{Cor}

The assertion about the H\"older continuity of the M\"obius Fourier
series $F(\mu;t)$ is stated (though in a rather different form) in a
recent paper of Veech \cite{veech2018} (see Remark 5.4).  The other
assertions of the corollary do not seem to be in the literature.

Several authors considered the more general exponential sums 
\begin{equation}
\label{eq:Skmu-def}
	S_k(\mu;x,t)=\sum_{1\le n\le x} \mu(n)e^{2\pi i n^k t},
\end{equation}
where $k$ is a positive integer. In particular, assuming the Generalized
Riemann Hypothesis, 
Zhan and Liu \cite{liu-zhan1996} showed that if $k\ge2$, 
then for any fixed $\epsilon>0$ and uniformly in $t$, one has 
\begin{equation}
	\label{eq:Skmu-bound}
	S_k(\mu;x,t)\ll_\epsilon  x^{\alpha_k+\epsilon}
\end{equation}
with $\alpha_k=1-2^{1-2k}$. 
Via Theorem \ref{thm:main-generalized}
and Corollary \ref{cor:main-generalized}
this translates into 
the following bounds for the H\"older exponent and dimensions of the
Fourier series
\begin{align}
	\label{eq:Moebius-series2}
	F_{k,k}(\mu; t)&=\sum_{n=1}^\infty\frac{\mu(n)e^{2\pi i n^k
	t}}{n^k}.
\end{align}


\begin{Cor}
\label{cor:Moebius2}
Let $k$ be an integer with $k\ge 2$.
Assume that the Generalized Riemann Hypothesis holds. Then we have:
\begin{itemize}
\item[(i)] The function
$F_{k,k}(\mu;t)$
is H\"older continuous with exponent $1-(1-2^{1-2k})/k-\epsilon$  
for any $\epsilon>0$, and the same holds  
for the absolute values and the real and imaginary parts of 
$F_{k,k}(\mu;t)$.
\item[(ii)]
The upper box-counting dimensions of the graphs
of the functions 
$|F_{k,k}(\mu;t)|$,
$\Re F_{k,k}(\mu;t)$,
and 
$\Im F_{k,k}(\mu;t)$
over the interval $[0,1]$ are 
bounded by $\leq 1+(1-2^{1-2k})/k$.
\end{itemize}
\end{Cor}

\begin{Remarks}
\mbox{}
\begin{enumerate}
\item
Similar  bounds on H\"older exponents and dimensions 
can be obtained from Theorem \ref{thm:main-generalized}
for the more general functions $\Fkp(\mu;t)=\sum_{n=1}^\infty
\mu(n)e^{2\pi i n^kt}/n^p$.  For example, if $p$ satisfies
$1<p\le k$, then, under the assumption of the Generalized Riemann
Hypothesis, $\Fkp(\mu;t)$ 
is H\"older continuous with exponent $(p-1+2^{1-2k})/k$.   
\item 
Since the function $\mu(n)$ is bounded by $1$, we have the trivial bound
\begin{equation}
	\label{eq:Skmu-bound-trivial}
	|S_k(\mu;x,t)|\le x\quad (x\ge1,\ t\in\R).
\end{equation}
Applying Theorem \ref{thm:main-generalized} with this bound
instead of \eqref{eq:Skmu-bound} yields, for $k\ge2$,
the conclusions of the corollary with H\"older exponent $1-1/k$ in place of
$1-(1-2^{1-2k})/k-\epsilon$ and dimension bound 
$1+1/k$ in place of $1+(1-2^{1-2k})/k$. These bounds can be regarded 
as ``baseline'' bounds for these quantities that  do not depend on the
oscillating nature of the M\"obius function.  

\item
In the other direction, 
the most optimistic estimate for $S_k(\mu;x,t)$ one can hope for is a 
squareroot bound analogous to \eqref{eq:S-mu-bound-conj}, i.e., 
a bound of the form
\begin{equation}
	\label{eq:Sk-mu-bound-conj}
	\sum_{1\le n\le x} \mu(n) e^{2\pi i n^kt}\ll_\epsilon  x^{1/2+\epsilon},
\end{equation}
for any fixed $\epsilon>0$ and uniformly in $t$. Under this assumption,  
Theorem \ref{thm:main-generalized} yields the assertions of the
corollary with H\"older exponent $1-1/(2k)-\epsilon$ and dimension bound 
$1+1/(2k)$. 
\end{enumerate}
\end{Remarks}

Table \ref{table:moebius-series} summarizes the H\"older exponents for
the series $F_{k,k}(\mu;t)$, $k=1,\dots,4$, obtained from the above
corollaries and the remarks following Corollary \ref{cor:Moebius2}.

\begin{table}[H]
\begin{center}
\renewcommand{\arraystretch}{1.8}
\begin{tabular}{|c||c|c|c|c|}
\hline
$k$& $1$ & $2$ &  $3$ & $4$
\\
\hline\hline
Unconditional 
& --- & $\frac12=0.5000$ & $\frac{2}{3}=0.6666\dots$ &
$\frac{3}{4}=0.7500$
\\
\hline
Assuming GRH
& $\frac{1}{4}=0.2500$  & $\frac{9}{16}=0.5625$ 
& $\frac{65}{96}=0.6770\dots$ & $\frac{385}{512}=0.7519\dots$
\\
\hline
Assuming \eqref{eq:Sk-mu-bound-conj} 
& $\frac{1}{2}=0.5000$ & $\frac34=0.7500$ 
& $\frac{5}{6}=0.8333\dots$ & $\frac{7}{8}=0.8750$
\\
\hline
\end{tabular}
\end{center}
\caption{H\"older exponents for the generalized M\"obius
series $F_{k,k}(\mu;t)=\sum_{n=1}^\infty \mu(n)e^{2\pi i n^k t}/n^k$,
under various assumptions. The exponents corresponding to the
values in the last two rows
are understood to be of the form $\eta-\epsilon$,
where $\eta$ is the value given in the table and $\epsilon$ is an
arbitrarily small positive number.}
\label{table:moebius-series}
\end{table}

Figures \ref{fig:moebius1}--\ref{fig:moebius44} below 
show how the fractal nature of the M\"obius series $F_{k,k}(\mu;t)$
becomes less and less prominent as $k$ gets larger.
This is consistent with the dimension bounds in part (ii) of Corollary
\ref{cor:Moebius2}, which approach $1$ as $k\to\infty$ .


\begin{figure}[H]
	\begin{subfigure}{.5\textwidth}
		\centering
		\includegraphics[width=.9\linewidth]{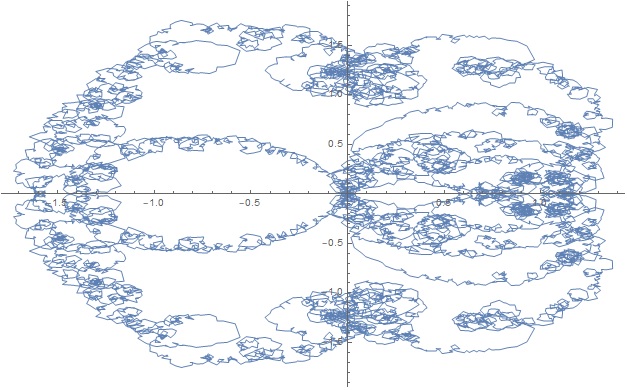}
		\caption{Path $F(\mu;t)$, $0\le t\le 1$}
		\label{fig:sfig1}
	\end{subfigure}%
	\begin{subfigure}{.5\textwidth}
		\centering
		\includegraphics[width=.9\linewidth]{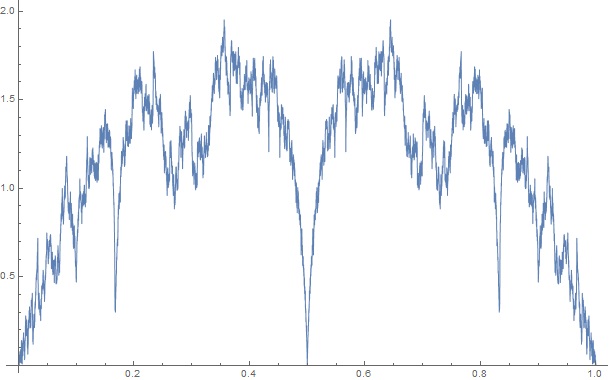}
		\caption{Graph of $|F(\mu;t)|$, $0\le t\le 1$}
		\label{fig:sfig2}
	\end{subfigure}
	\caption{The Fourier series $F(\mu;t)=\sum_{n=1}^\infty
	\mu(n)e^{2 \pi i n t}/n$.}
	\label{fig:moebius1}
\end{figure}

\begin{figure}[H]
	\begin{subfigure}{.5\textwidth}
		\centering
		\includegraphics[width=.9\linewidth]{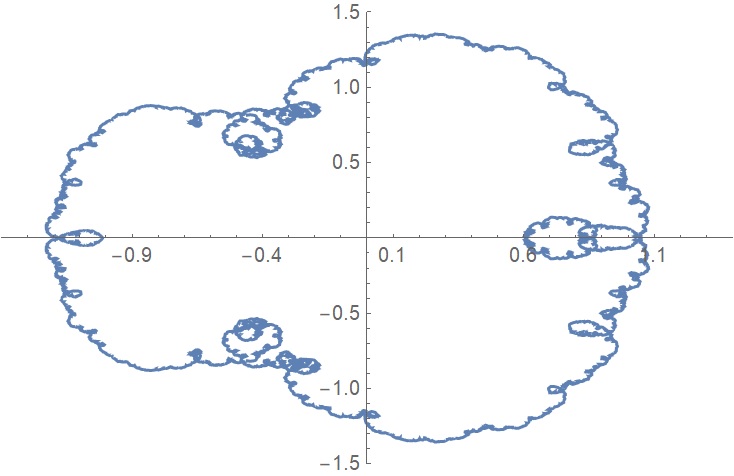}
		\caption{Path $F_{2,2}(\mu;t)$, $0\le t\le 1$}
	\end{subfigure}%
	\begin{subfigure}{.5\textwidth}
		\centering
		\includegraphics[width=.9\linewidth]{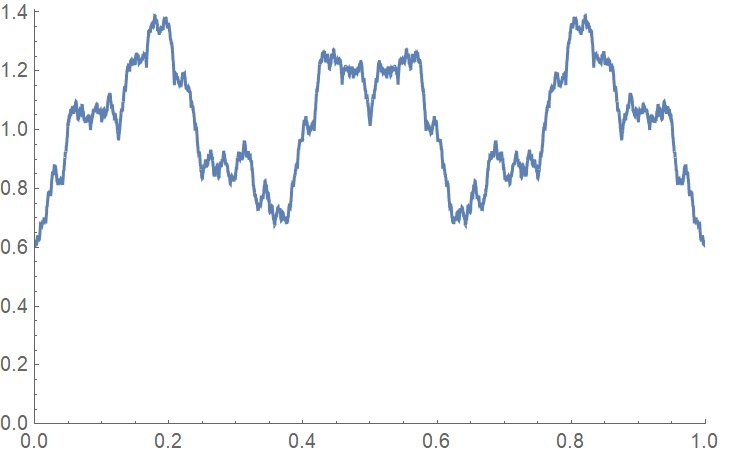}
		\caption{Graph of $|F_{2,2}(\mu;t)|$, $0\le t\le 1$}
	\end{subfigure}
	\caption{The Fourier series $F_{2,2}(\mu;t)=\sum_{n=1}^\infty
	\mu(n)e^{2 \pi i n^2 t}/n^2$.}
	\label{fig:moebius22}
\end{figure}

\begin{figure}[H]
	\begin{subfigure}{.5\textwidth}
		\centering
		\includegraphics[width=.9\linewidth]{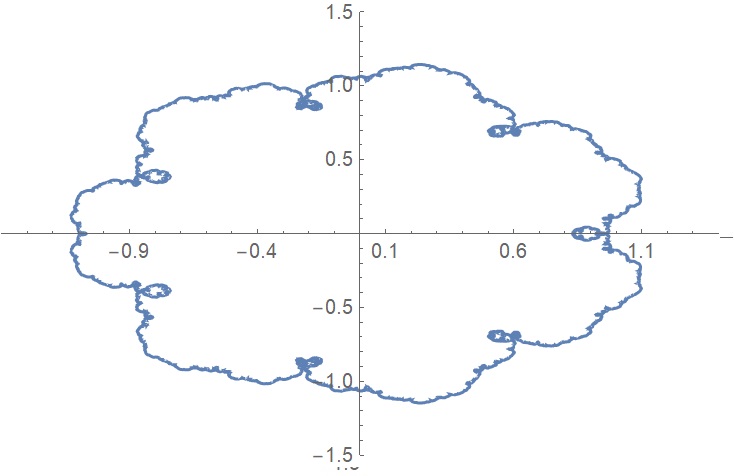}
		\caption{Path $F_{3,3}(\mu;t)$, $0\le t\le 1$}
	\end{subfigure}%
	\begin{subfigure}{.5\textwidth}
		\centering
		\includegraphics[width=.9\linewidth]{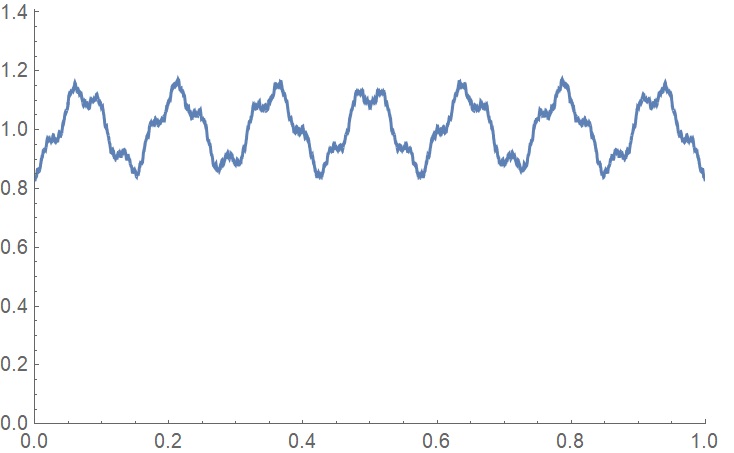}
		\caption{Graph of $|F_{3,3}(\mu;t)|$, $0\le t\le 1$}
	\end{subfigure}
	\caption{The Fourier series $F_{3,3}(\mu;t)=\sum_{n=1}^\infty
	\mu(n)e^{2 \pi i n^3 t}/n^3$.}
	\label{fig:moebius33}
\end{figure}
\begin{figure}[H]
	\begin{subfigure}{.5\textwidth}
		\centering
		\includegraphics[width=.9\linewidth]{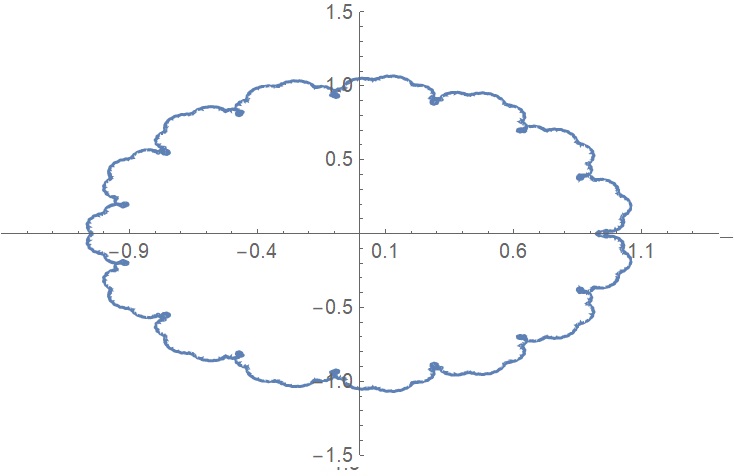}
		\caption{Path $F_{4,4}(\mu;t)$, $0\le t\le 1$}
	\end{subfigure}%
	\begin{subfigure}{.5\textwidth}
		\centering
		\includegraphics[width=.9\linewidth]{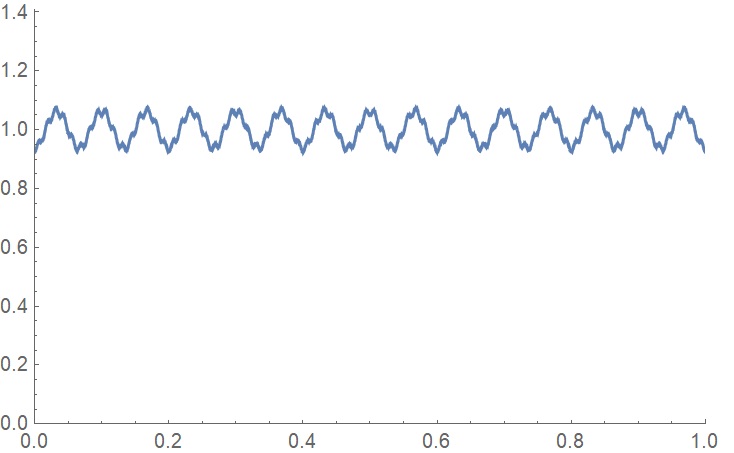}
		\caption{Graph of $|F_{4,4}(\mu;t)|$, $0\le t\le 1$}
	\end{subfigure}
	\caption{The Fourier series $F_{4,4}(\mu;t)=\sum_{n=1}^\infty
	\mu(n)e^{2 \pi i n^4 t}/n^4$.}
	\label{fig:moebius44}
\end{figure}


\section{Concluding Remarks}
\label{sec:conclusion}

We conclude this paper by discussing some possible extensions of our
results and some related questions and open problems.

\subsection{Localized versions of Theorems \ref{thm:main} and
\ref{thm:main-generalized}.}
A key aspect of Theorems \ref{thm:main} and \ref{thm:main-generalized}
is the uniformity in $t$ of both the bound \eqref{eq:S-bound} for
$S(x,t)$ and the bound \eqref{eq:F-bound} for $F(t+h)-F(t)$.  This
raises the question whether these theorems can be localized, in the sense
that assuming a bound for the exponential sums $S(x,t)$ at a
\emph{particular}
point $t$ yields a bound for the \emph{local} (or \emph{pointwise}) 
H\"older exponent   of
the associated Fourier series at that point (see Definition
\ref{def:holder}). In other words, if we have a bound for $S(x,t)$ of the form
\eqref{eq:S-bound}, but with the constant $C=C(t)$ and the exponent
$\alpha=\alpha(t)$ being allowed to depend on $t$, can we obtain
non-trivial bounds for the pointwise H\"older exponents $\eta(t)$ of
the function $F(t)$?  

A result of this type, if true, would have significant applications.
For example, in the case of the M\"obius function $\mu(n)$, Murty and
Sankaranarayanan \cite{murty-sanka2002} proved \emph{unconditionally}
(i.e., without assuming the Generalized Riemann Hypothesis)  a bound of
the form $S(\mu;x,t)\ll_\epsilon x^{4/5+\epsilon}$ for a certain class of
numbers $t$ that includes all algebraic irrational numbers.  If a
localized version of Theorem \ref{thm:main} were available, such a bound
would imply non-trivial bounds on the local H\"older exponent of the
M\"obius Fourier series $F(\mu;t)$ for the same class of numbers $t$.

Our method relies on the uniformity of the exponential sum bound
\eqref{eq:S-bound} in an essential manner and does not seem to be
capable of yielding a localized version of the type mentioned. 
The most we can prove in this direction
 is that an exponential
sum bound that is uniform \emph{in some open neighborhood of $t$}
implies a corresponding bound for the H\"older exponent in that same
neighborhood. This, however, would not be sufficient for the application
to the M\"obius Fourier series mentioned above. 

The question whether there is a  local analog of Theorem
\ref{thm:main} that relates estimates for $S(x,t)$ at a
\emph{particular} point $t$ to estimates for the \emph{local} or
\emph{pointwise} H\"older exponent of $F(t)$ at the same point $t$
remains open.  We expect, however, that such a result, if true, would
not hold in nearly the same generality as Theorem \ref{thm:main}.
Some interesting partial results in this direction have recently  
been obtained by Chamizo, Petrykiewicz, and Ruiz-Cabelo
\cite{chamizo-petrykiewicz-ruiz2017}.

\subsection{Lower bounds on the H\"older exponent and bi-H\"older
continuity.}\label{subsec:bi-Holder}
Theorem \ref{thm:main} yields, under the assumption \eqref{eq:S-bound},
a bound of the form 
\begin{equation}
\label{eq:F-upper-holder}
|F(t+h)-F(t)|\le C_1 h^\eta
\end{equation}
uniformly in $t,h\in(0,1)$. A natural question is whether a similar
bound in the other direction holds, i.e., whether one has, with suitable
constants $C_1'>0$ and $\eta'>0$, 
\begin{equation}
\label{eq:F-lower-holder}
|F(t+h)-F(t)|\ge C_1' h^{\eta'}
\end{equation}
uniformly in $t,h\in(0,1)$.  A function satisfying both 
\eqref{eq:F-upper-holder}
and \eqref{eq:F-lower-holder} is called bi-H\"older continuous.

Lower bounds of the form \eqref{eq:F-lower-holder} (or localized
versions of such bounds) have been established for certain narrowly
defined  classes of functions such as the Weierstrass function 
and functions of Riemann type.
However, for more general classes of functions such as
the functions considered in Theorem \ref{thm:main} very little is known.
In particular, whether the M\"obius Fourier series $F(\mu;t)$ satisfies 
a non-trivial lower bound of the form \eqref{eq:F-lower-holder} is 
not known.

\subsection{The constants in Theorems \ref{thm:main} and
\ref{thm:main-generalized}.}
As we have noted, the exponent $1-\alpha$ in the H\"older estimate
\eqref{eq:F-bound} of Theorem \ref{thm:main} is optimal. This
raises the question to what extent the constant $C_1=C_1(C,\alpha)$ 
in this theorem is also optimal. A small improvement can be obtained by
keeping the factor $(2\pi + 4\pi/\alpha)$ in \eqref{eq:lem2c} instead of
replacing it by the larger value $6\pi/\alpha$. This results in the 
value 
\[
C_1(C,\alpha)=
\frac{8\alpha+2\pi(\alpha+2)(1-\alpha)}{\alpha(1-\alpha)}C
\]
for the constant in Theorem \ref{thm:main}, 
which is slightly smaller than the value \eqref{eq:C1-formula}
given in that theorem though has similar asymptotic behavior as 
$\alpha\to 0+$ or $\alpha\to 1-$.  Is this constant best-possible, at
least as far as its dependence on $C$ and $\alpha$ is concerned?

Since scaling the coefficients $f(n)$ by a fixed constant scales the
associated exponential sum $S(x,t)$ and Fourier series $F(t)$ by the
same constant, it is obvious that $C_1$ must depend linearly on the
constant $C$ in \eqref{eq:S-bound}.  On the other hand, the dependence
on $\alpha$ is less clear.  In particular, one can ask whether, as
$\alpha$ approaches $0$ or $1$, the constant $C_1$ is \emph{necessarily}
unbounded.  With our approach this seems unavoidable: Lemmas
\ref{lem:cauchy-estimate} and \ref{lem:FN-diff} introduce factors
proportional to $1/(1-\alpha)$ and $1/\alpha$, respectively, which in
turn forces the constant $C_1$ in Theorem \ref{thm:main} to depend on
$\alpha$ at a rate roughly proportional to $1/(\alpha(1-\alpha))$.

Similar questions can be asked about the constant $C_2$ in Theorem
\ref{thm:main-generalized} and its dependence on the parameters $k$,
$p$, $C$, and $\alpha$.

\subsection{Lower dimension bounds.}
Theorem \ref{thm:main} provides \emph{upper} bounds on the 
box-counting dimensions of the path 
$F([0,1])$ and the graph $G(F)$.
In the case of the Weierstrass function, these bounds
are known to be sharp, i.e., they also represent lower bounds on the
dimension.  One can ask whether non-trivial lower bounds on the
dimensions of $F([0,1])$ and $G(F)$ can be obtained in the setting 
of Theorem \ref{thm:main}. More precisely, what non-trivial conditions can be assumed on $f$ or $S(f;x,t)$ in order to deduce a corresponding 
non-trivial lower bound on the dimensions of $F([0,1])$ and $G(F)$? Note the connection of this question to the discussion in Section \ref{subsec:bi-Holder} in view of $(3.11)$ in \cite{fraser2021} and Corollary 11.2 in \cite{falconer2014}, which imply lower dimension bounds for $F([0,1])$ and $G(F)$, respectively, if a bi-H\"older condition holds for $F$.

\subsection{Bounds on Hausdorff and Assouad dimensions.}
Corollaries \ref{cor:main} and \ref{cor:main-generalized} give bounds on
the \emph{box-counting} dimension, $\dim_B$, of graphs associated with
the functions $F(t)$ and $\Fkp(t)$. Other notions of dimension
that have been studied in the literature include the \emph{Hausdorff
dimension}, denoted by $\dim_H$, and the \emph{Assouad dimension}, 
denoted by $\dim_A$; see \cite[p.~6]{fraser2021} and
\cite[p.~10]{fraser2021} for precise definitions of these concepts.  

It is known (see \cite[Lemma 2.4.3]{fraser2021})
that, for any bounded set $E\subset \R^d$, 
\begin{equation}
\label{eq:dimineq1}
	\dim_H E \leq \underline{\dim_B} E 
	\leq \overline{\dim_B}E
	\leq \dim_AE.
\end{equation}
Thus any upper bound on the box-counting dimension is also an upper
bound on the Hausdorff dimension.  It follows that the upper bounds for
the box-counting dimension provided by Corollaries \ref{cor:main} and
\ref{cor:main-generalized} also hold for the Hausdorff dimension.

On the other hand, \eqref{eq:dimineq1} also shows that the Assouad dimension is
bounded \emph{below} by the (upper) box-counting dimension, so 
no similar conclusion can be drawn for the Assouad dimension. 
Moreover, there is no analog of the
inequalities of Proposition \ref{prop:dimdist} for the Assouad dimension.
That is, there is no non-trivial bound for the Assouad dimension 
of the path $F([0,1])$ and the graph $G(F)$
in terms of the H\"older exponent of the
function $F(t)$.  It would be interesting to know if one can obtain non-trivial
bounds on the Assouad dimension under assumptions similar to those of
Corollary \ref{cor:main}. 
Even in the special case of the Weierstrass function,
determining the exact value of the Assouad dimension of its graph is
still an open problem (see \cite{fraser2021}, Question 17.11.1).

\subsection{Random Fourier series.} A natural model for Fourier series
with pseudo-random coefficients such as the M\"obius and Liouville
functions $\mu(n)$ and $\lambda(n)$ is a random series of the form
\begin{equation}
\label{eq:FX-def}
F(X;t)=\sum_{n=1}^\infty \frac{X_n e^{2\pi i nt}}{n},
\end{equation}
with coefficients given by a sequence of independent random variables
$X_n$ that take on values $+1$ and $-1$ with probability $1/2$ each.
The behavior (with probability $1$) of such random series can serve as a
source for conjectures for deterministic series such as those associated
with the M\"obius and Liouville functions.

Random Fourier series of this type have been investigated in the
literature.  In particular, a result of Kahane (see Theorem 3 in Chapter
7 of \cite{kahane1985}) shows that, with probability $1$, the series
\eqref{eq:FX-def} is H\"older continuous with exponent $1/2$, and that
this exponent is best-possible.  Thus it seems reasonable to conjecture 
that the M\"obius series $F(\mu;t)$ is H\"older continuous 
with exponent  $1/2-\epsilon$, for any fixed $\epsilon$,
and that the constant $1/2$ here is best possible.  

Recently, Kowalski and Sawin \cite{kowalski-sawin2016} considered series
of the form \eqref{eq:FX-def}, where the $X_n$ are independent random
variables having the \emph{Sato-Tate distribution}. They showed that
this particular random series arises naturally in number theory as a
limiting   process of random functions associated with Kloosterman sums,
and they established a variety of properties of this series. 






\providecommand{\bysame}{\leavevmode\hbox to3em{\hrulefill}\thinspace}
\providecommand{\MR}{\relax\ifhmode\unskip\space\fi MR }

\end{document}